\documentclass[12pt]{amsart}
\usepackage{amssymb,amsmath,amsthm}
\numberwithin{equation}{section}
\newtheorem{theorem}{Theorem}[section]
\newtheorem{lemma}[theorem]{Lemma}
\newtheorem{remark}[theorem]{Remark}

\begin{document}

\newcommand{\E}{\mathbb{E}}
\renewcommand{\P}{\mathbb{P}}
\newcommand{\R}{\mathbb{R}}
\renewcommand{\a}{\alpha}
\renewcommand{\b}{\beta}
\newcommand{\ep}{\varepsilon}
\newcommand{\tep}{\tilde\varepsilon}
\newcommand{\hep}{\hat\varepsilon}
\newcommand{\noin}{\noindent}
\newcommand{\De}{\Delta}
\newcommand{\la}{\lambda}
\newcommand{\La}{\Lambda}
\newcommand{\de}{\delta}
\newcommand{\tgl}{\tilde\lambda}
\newcommand{\si}{\sigma}
\newcommand{\Ga}{\Gamma}
\newcommand{\ga}{\gamma}
\newcommand{\Si}{\Sigma}
\renewcommand{\o}{\omega}
\renewcommand{\O}{\Omega}
\newcommand{\ol}{\overline}
\newcommand{\bXi}{{\mbox{\boldmath$\Xi$}}}
\newcommand{\bbxi}{{\mbox{\boldmath$\xi$}}}
\newcommand{\bbphi}{{\mbox{\boldmath$\phi$}}}
\newcommand{\bbPhi}{{\mbox{\boldmath$\Phi$}}}
\newcommand{\bbep}{{\mbox{\boldmath$\ep$}}}
\newcommand{\bgS}{{\mbox{\boldmath$\Sigma$}}}
\newcommand{\bga}{{\mbox{\boldmath$\alpha$}}}
\newcommand{\bbgd}{{\mbox{\boldmath$\delta$}}}
\newcommand{\bbtau}{{\mbox{\boldmath$\tau$}}}
\newcommand{\bbPsi}{{\mbox{\boldmath$\Psi$}}}
\newcommand{\bgma}{{\mbox{\boldmath$\gamma$}}}
\newcommand{\bbgG}{{\mbox{\boldmath$\Gamma$}}}
\newcommand{\btau}{{\mbox{\boldmath$\tau$}}}
\newcommand{\brho}{{\mbox{\boldmath$\rho$}}}
\newcommand{\bbgs}{{\mbox{\boldmath$\gs$}}}
\newcommand{\bbeta}{{\mbox{\boldmath$\beta$}}}
\newcommand{\bbzeta}{{\mbox{\boldmath$\zeta$}}}
\newcommand{\bbpsi}{{\mbox{\boldmath$\psi$}}}
\newcommand{\bbtheta}{{\mbox{\boldmath$\theta$}}}
\newcommand{\bdeta}{{\mbox{\boldmath$\eta$}}}
\newcommand{\bgl}{{\mbox{\boldmath$\lambda$}}}
\newcommand{\bbell}{{\mbox{\boldmath$\ell$}}}
\newcommand{\tell}{\tilde\ell}
\newcommand{\tbbell}{\widetilde{\mbox{\boldmath$\ell$}}}
\newcommand{\bzeta}{{\mbox{\boldmath$\zeta$}}}
\newcommand{\rtr}{{\rm tr}}
\newcommand{\Cov}{{\rm Cov}}
\newcommand{\rk}{{\rm rank}}
\newcommand{\sgn}{{\rm sgn}}
\newcommand{\ctan}{{\rm ctan}}
\newcommand{\rank}{{\rm rank}}
\newcommand{\rE}{{\rm E}}
\newcommand{\RRe}{{\rm Re}}
\newcommand{\IIm}{{\rm Im}}
\newcommand{\txi}{\tilde\xi}
\newcommand{\bxi}{{\mbox{\boldmath$\xi$}}}
\newcommand{\fii}{\frac{1}{2}}
\newcommand{\fnn}{\frac1{\sqrt n}}
\newcommand{\fn}{\frac{1}{n}}
\newcommand{\siln}{\sum_{i=1}^n}
\newcommand{\skln}{\sum_{k=1}^n}
\newcommand{\nskln}{\frac{1}{n}\sum_{k=1}^n}
\newcommand{\nsjln}{\frac{1}{n}\sum_{j=1}^n}
\newcommand{\sjln}{\sum_{j=1}^n}
\newcommand{\sklp}{\sum_{k=1}^p}
\newcommand{\psklp}{\frac{1}{p}\sum_{k=1}^p}
\newcommand{\nin}{\not\in}
\newcommand{\non}{\nonumber\\}
\newcommand{\ox}{{\overline x}}
\newcommand{\oy}{{\overline y}}
\newcommand{\hx}{{\widehat x}}
\newcommand{\oF}{{\overline F}}
\newcommand{\oX}{{\overline X}}
\newcommand{\obX}{{\overline \bfX}}
\newcommand{\oW}{{\overline \bXi}}
\newcommand{\omu}{{\overline \mu}}
\newcommand{\bbmu}{{\mbox{\boldmath$\mu$}}}
\newcommand{\bbnu}{{\mbox{\boldmath$\nu$}}}
\newcommand{\bpsi}{{\mbox{\boldmath$\Psi$}}}
\newcommand{\bsigma}{{\mbox{\boldmath$\sigma$}}}
\newcommand{\oox}{{\overline{\overline x}}}
\newcommand{\ooW}{{\overline{\overline \bXi}}}
\newcommand{\oomu}{{\overline{\overline \mu}}}
\newcommand{\tomu}{{\widetilde{\overline \mu}}}
\newcommand{\tox}{{\widetilde{\overline x}}}
\newcommand{\toXi}{{\widetilde{\overline \bXi}}}
\newcommand{\darrow}{\stackrel{\cal D}\to}
\def\iparrow{\buildrel i.p.\over\longrightarrow}
\def\asarrow{\buildrel a.s.\over\longrightarrow}
\newcommand{\tb}{{\tilde \bfb}}
\newcommand{\tx}{{\tilde x}}
\newcommand{\tnab}{{\tilde \nabla}}
\newcommand{\tXi}{{\tilde \bXi}}
\newcommand{\diag}{{\rm diag}}
\newcommand{\rP}{{\rm P}}
\newcommand{\rVar}{{\rm Var}}
\newcommand{\rCov}{{\rm Cov}}
\newcommand{\bfA}{{\bf A}}
\newcommand{\bfa}{{\bf a}}
\newcommand{\bfB}{{\bf B}}
\newcommand{\bfb}{{\bf b}}
\newcommand{\bfC}{{\bf C}}
\newcommand{\bfc}{{\bf c}}
\newcommand{\bfD}{{\bf D}}
\newcommand{\bfd}{{\bf d}}
\newcommand{\bfe}{{\bf e}}
\newcommand{\bfE}{{\bf E}}
\newcommand{\bff}{{\bf f}}
\newcommand{\bfF}{{\bf F}}
\newcommand{\bfg}{{\bf g}}
\newcommand{\bfG}{{\bf G}}
\newcommand{\bfH}{{\bf H}}
\newcommand{\bfh}{{\bf h}}
\newcommand{\bfI}{{\bf I}}
\newcommand{\bfi}{{\bf i}}
\newcommand{\bfj}{{\bf j}}
\newcommand{\bfJ}{{\bf J}}
\newcommand{\bfk}{{\bf k}}
\newcommand{\bfK}{{\bf K}}
\newcommand{\bfl}{{\bf 1}}
\newcommand{\bfL}{{\bf L}}
\newcommand{\bfM}{{\bf M}}
\newcommand{\bfm}{{\bf m}}
\newcommand{\bfn}{{\bf n}}
\newcommand{\bfN}{{\bf N}}
\newcommand{\bfQ}{{\bf Q}}
\newcommand{\bfp}{{\bf p}}
\newcommand{\bfP}{{\bf P}}
\newcommand{\bfq}{{\bf q}}
\newcommand{\bfO}{{\bf O}}
\newcommand{\bfR}{{\bf R}}
\newcommand{\bfr}{{\bf r}}
\newcommand{\bfs}{{\bf s}}
\newcommand{\bfS}{{\bf S}}
\newcommand{\bfY}{{\bf Y}}
\newcommand{\bfy}{{\bf y}}
\newcommand{\bfZ}{{\bf Z}}
\newcommand{\bfz}{{\bf z}}
\newcommand{\tbbS}{\widetilde{\bf S}}
\newcommand{\hbbS}{\widehat{\bf S}}
\newcommand{\obbS}{\overline{\bf S}}
\newcommand{\bfT}{{\bf T}}
\newcommand{\bft}{{\bf t}}
\newcommand{\hbbT}{\widehat{\bf T}}
\newcommand{\tbbT}{\widetilde{\bf T}}
\newcommand{\obT}{{\overline{\bf T}}}
\newcommand{\bfU}{{\bf U}}
\newcommand{\bfu}{{\bf u}}
\newcommand{\bfV}{{\bf V}}
\newcommand{\bfv}{{\bf v}}
\newcommand{\tbbv}{\widetilde{\bf v}}
\newcommand{\bfw}{{\bf w}}
\newcommand{\bfW}{{\bf W}}
\newcommand{\tbbW}{\widetilde{\bf W}}
\newcommand{\hbbB}{\widehat{\bf B}}
\newcommand{\hbbW}{\widehat{\bf W}}
\newcommand{\bfX}{{\bf X}}
\newcommand{\tbbB}{\widetilde {\bf B}}
\newcommand{\tbbX}{\widetilde {\bf X}}
\newcommand{\hbbX}{\widehat {\bf X}}
\newcommand{\bfx}{{\bf x}}
\newcommand{\obbx}{{\overline{\bf x}}}
\newcommand{\tbbx}{{\widetilde{\bf x}}}
\newcommand{\hbbx}{{\widehat{\bf x}}}
\newcommand{\obby}{{\overline{\bf y}}}
\newcommand{\hbbY}{{\widehat{\bf Y}}}
\newcommand{\tbbY}{{\widetilde{\bf Y}}}
\newcommand{\tbby}{{\tilde{\bf y}}}
\newcommand{\bbI}{{\mathbb I}}
\newcommand{\cC}{{\cal C}}
\newcommand{\cD}{{\cal D}}
\newcommand{\cE}{{\cal E}}
\newcommand{\cF}{{\cal F}}
\newcommand{\cG}{{\cal G}}
\newcommand{\cH}{{\cal H}}
\newcommand{\cI}{{\cal I}}
\newcommand{\cM}{{\cal M}}
\newcommand{\cP}{{\cal P}}
\newcommand{\cQ}{{\cal Q}}
\newcommand{\cR}{{\cal R}}
\newcommand{\cS}{{\cal S}}
\newcommand{\cT}{{\cal T}}
\newcommand{\cX}{{\cal X}}
\newcommand{\cU}{{\cal U}}
\newcommand{\cV}{{\cal V}}
\newcommand{\Th}{{\bfT^{1/2}}}
\newcommand{\um}{{\underline m}}
\newcommand{\us}{{\underline s}}
\newcommand{\umn}{{\underline m}_n}
\newcommand{\ua}{{\underline a}}
\newcommand{\uF}{{\underline F}}
\newcommand{\gma}{\gamma}

\renewcommand{\(}{\left(}
\renewcommand{\)}{\right)}
\newcommand{\lj}{\left|}
\newcommand{\rj}{\right|}
\newcommand{\lb}{\label}
\newcommand{\no}{\nonumber}
\newcommand{\hi}{ H\"{o}lder's inequality }

\title[Convergence Rates]{A Note  on  Rate of Convergence in Probability to Semicircular Law}

\begin{abstract}
In the present paper, we prove that under the assumption of the finite  sixth moment for elements of a Wigner matrix, the convergence rate of its empirical spectral distribution to the Wigner semicircular  law in probability is  $O(n^{-1/2})$ when the dimension $n$ tends to infinity.
\end{abstract}
\author{Zhidong Bai}
\address{KLASMOE and School of Mathematics \& Statistics, Northeast Normal University, Changchun, 130024, P.R.C.\\
and Department of Statistics and Applied Probability, National University of Singapore}
\email{baizd@nenu.edu.cn}

\author{Jiang Hu}
\address{KLASMOE and School of Mathematics \& Statistics, Northeast Normal University, Changchun, 130024, P.R.C..}
\email{huj156@nenu.edu.cn}

\author{Guangming Pan}
\address{Division of Mathematical Sciences, School of Physical and Mathematical Sciences, Nanyang Technological University.}
\email{gmpan@ntu.edu.sg}

\author{Wang Zhou}
\address{Department of Statistics and Applied Probability, National University of Singapore, Singapore 117546}
\thanks{Z. D. Bai was partially supported by CNSF 10871036. J. Hu was partially supported by the Fundamental Research Funds for the Central Universities 10ssxt149.
 W. Zhou  was partially supported  by grant
          R-155-000-106-112  at
the National University of Singapore}
\email{stazw@nus.edu.sg}

\subjclass{Primary   60F15; Secondary 62H99}

\keywords{convergence rate, Wigner matrix, Semicircular Law,  spectral distribution}

\maketitle
\section{Introduction and the result.}
A Wigner matrix  $\bfW_n=n^{-1/2}\(x_{ij}\)_{i,j=1}^n$ is defined to be  a  Hermitian random matrix whose  entries on and above the diagonal are independent zero-mean random variables.   It is an important  model for depicting heavy-nuclei atoms, which begin with the seminal work of Wigner in 1955 (\cite{Wigner1955}). Details in this area can be found in \cite{Mehta2004}.

There are various mathematical tools in the study of Wigner matrices in the past half century (see \cite{AndersonG2010}). One of the most popular instruments is the limit theory of  empirical spectral distribution (ESD). Here, for any $n\times n$ matrix $\bfA$ with real eigenvalues, the ESD of $\bfA$ is defined by
\begin{align*}
    F^{\bfA}(x)=\frac{1}{n}\sum_{i=1}^nI(\la_i^{\bfA}\le x),
\end{align*}
where $\la_i^{\bfA}$ denotes the $i$-th smallest eigenvalue of $\bfA$ and $I(B)$  denotes the indicator function of an event $B$. It is proved that,under  assumptions of  for all $i,j$, $\E|x_{ij}|^2=\si^2$, the ESD $F^{\bfW_n}(x)$ converges almost surely  to a non-random distribution $F(x)$ which has the destiny function
\begin{align}
    f(x)=\frac{1}{2\pi\si}\sqrt{4\si^2-x^2},~~x\in[-2\si,2\si].\label{01}
\end{align}This is also known as the Wigner semicircular law (see \cite{Wigner1955}, \cite{BaiS2010}).

The rate of convergence  is  important  in establishing the central limit theorem for linear spectral statistics of Wigner matrices (\cite{BaiW2009,BaiS2010}).  There are some partial results in this area.  In \cite{Bai1993},   Bai  proved that under the assumption of $\sup_n\sup_{i,j}\E x_{ij}^4<\infty$, the rate of
$$\De_n=\| \E F^{\mathbf{W}_n}-F\|:=\sup_{x}| F^{\mathbf{W}_n}(x)-F(x)|$$
 tending to 0  is $O(n^{-1/4})$.  Bai et al. in \cite{BaiM1997} obtained  that the rate established in \cite{Bai1993} is still valid for
$$\De_p= \| F^{\mathbf{W}_n}-F\|:=\sup_{x}| F^{\mathbf{W}_n}(x)-F(x)|$$ Under a stronger condition that $\sup_n\sup_{i,j}\E x_{ij}^8<\infty$,  Bai et al. in \cite{BaiM2002} showed  that $\De_n=O(n^{-1/2})$ and $\De_p=O_p(n^{-2/5})$ (Bai  and  Silverstein improve this condition up to $\sup_n\sup_{i,j}\E x_{ij}^6<\infty$ in their book \cite{BaiS2010} ).  Later,  G\"{o}tze et al. in \cite{GotzeT2003} derived  $\De_n=O(n^{-1/2})$ as well assuming  fourth moment, and $\De_p=O_p(n^{-1/2})$  at the cost of the twelfth moment of the matrix entries. There are some other results with some special  assumptions on the matrix entries. For which one can  refer to \cite{GotzeT2005,GotzeT2007,Tikhomirov2009,BobkovG2010a}.

In this note we prove that the twelfth moment condition in \cite{GotzeT2003} could be  reduced to the sixth the moment assumption when getting $\De_p=O_p(n^{-1/2})$. 
Our main result of this paper is as follow.
\begin{theorem}\label{th1}
Assume that
\begin{itemize}
  \item $\E x_{ij}=0,\mbox{ for all }1\le i\le j\le n,$
  \item $\E |x_{ii}^2|=\si^2>0, \E |x_{ij}|^2=1, \mbox{ for all }1\le i< j\le n,$
  \item $ \sup_n\sup_{1\leq i<j\leq n}\E |x_{ii}^3|,\E |x_{ij}|^{6}<\infty.$
\end{itemize}
Then we have
 \begin{align}\label{1}
\De_p:=\| F^{\mathbf{W}_n}-F\|=O_p(n^{-1/2}).
 \end{align}
\end{theorem}
\begin{remark}
It is not clear what the exact rate and the optimal conditions are. As far as we know, the best known rate in the literature is $O(n^{-1/2})$.
\end{remark}


  The rest of this paper is organized as follows. The  main tool of proving the theorem is introduced in Section 2.  Theorem \ref{th1} is  proved in Section 3 and some technical lemmas are given in Section 4. Throughout this paper, constants appearing in inequalities are represented by $C$ which are nonrandom and may take different values
from one appearance to another.

\section{The main tool} Our main tool to prove the theorem is  a Berry-Esseen type  inequality  in \cite{Bai1993}.
\begin{lemma}(Bai inequality)\label{bai}
Let $F$ be a distribution function  and let $G$ be a function of bounded variation satisfying $\int|F(x)-G(x)|dx<\infty$. Denote their Stieltjes transforms by $s_{F}(z)$ and
$s_{G}(z)$ respectively, where $z=u+iv\in\mathbb{C}^{+}$. Then
\begin{align*}
\|F-G\|&\leq\frac{1}{\pi(1-\zeta)(2\rho-1)}\left(\int_{-A}^A|s_{F}(z)-s_{G}(z)|du\right.\\
&\quad+2\pi v^{-1}\int_{|x|>B}|F(x)-G(x)|dx\nonumber\\
&\qquad\left.,+v^{-1}\sup_x\int_{|u|\leq2v\epsilon}|G(x+u)-G(x)|du\right),
\end{align*}
where the constants  $A>B>0$, $\zeta$ and $\epsilon$ are restricted by
$ \rho=\frac{1}{\pi}\int_{|u|\leq
\epsilon}\frac{1}{u^2+1}du>\frac{1}{2}$, and  $
\zeta=\frac{4B}{\pi(A-B)(2\rho-1)}\in(0,1). $
\end{lemma}

Here we should notice that we can use the same methods in \cite{GotzeT2003} to prove our theorem. However,  G\"{o}tze-Tikhomirov inequality (see Corollary 2.3 in \cite{GotzeT2003}) involves the supremum of $|s_n(z)-\E s_n(z)|$ over $\Im z$ in some interval. This makes the proof rather complicated. Therefore in this paper, we use Bai inequality instead of G\"{o}tze-Tikhomirov inequality which could make the presentation simpler.

\section{The proof  of  Theorem \ref{th1}.}
We will firstly introduce  a new technique which can handle the moment conditions efficiently. That is given in Lemma \ref{3.2}. Then, by using this lemma and dividing the expression of $\E|s_n-\E s_n|^2$, we prove our theorem step by step.

Before proving the theorem, we  introduce  some  notation. Denote $\bfI_n$ be  the identity matrix of size $n$ and $\bfa_i$ be the $i$th column of $\bfW_n$ with $x_{ii}$ removed. Define $\bfD(z)=n^{-1/2}\mathbf{W}_n-z\bfI_n$, $\bfD_i(z)=\bfD(z)-n^{-1}\bfa_i\bfa_i^*$ and $s_n=s_n(z)=s_{F^{\bfW_n}}(z)$.  Moreover write
\begin{gather*}
\b_i=\(n^{-1/2}x_{ii}-z-n^{-1}\bfa_i^*\bfD^{-1}_i\bfa_i\)^{-1},\quad
\ga_i=\bfa_i^*\bfD_i^{-1}\bfa_i-tr\bfD_i^{-1}\\
\ep_i=n^{-1/2}x_{ii}-n^{-1}\bfa_i^*\bfD_i^{-1}\bfa_i+ \E s_n(z),\quad
\hat\ga_i=\bfa_i^*\bfD_i^{-2}\bfa_i-tr\bfD_i^{-2}\\
\xi_i=tr\bfD^{-1}-tr\bfD_i^{-1},\quad
a_n=(z+\E s_n(z))^{-1},\quad b_n=(z+2\E s_n(z))^{-1}
\end{gather*}

Throughout this section, we denote $z=u+iv$, $u\in[-16,16]$ and $1\geq v\geq v_0= C_0n^{-1/2}$ with an appropriate constant $C_0$.
Let $s=s(z)=s_{F}(z)$, we know that  (see (3.2) in \cite{Bai1993} )
\begin{align*}
    s(z)=-\frac{1}{2}\(z-\sqrt{z^2-4}\) \mbox{ for all } z\in\mathbb C^+.
\end{align*}
Then we have
\begin{align}
    \int_{-16}^{16} \frac{1}{|z+2s(z)|}du\leq\int_{-16}^{16} \frac{1}{\sqrt{|z^2-4|}}du\le \int_{-16}^{16} \frac{1}{\sqrt{|u^2-4|}}du<10\label{z}.
\end{align}
In addition, by Lemma \ref{bai} and Theorem 8.2 in \cite{BaiS2010}, we have for some positive constant $C$,
\begin{align}
    \E\|F^{\bfW_n}-F\|\leq C \int_{-16}^{16}\E |s_n(z)-\E s_n(z)|du+O(n^{-1/2})\label{f}.
\end{align}
Therefore, the rest of the  proof is reduced to the lemma below.
\begin{lemma}\label{l1} Under the assumptions in Theorem \ref{th1}, for any $1>v\geq v_0=C_0n^{-1/2}$ with sufficiently large $C_0>0$, we have
  \begin{align*}
    \E \lj s_n(z)-\E s_n(z) \rj^2\leq\frac{C}{n|z+2s(z)|^2}.
  \end{align*}
\end{lemma}
\subsection{Known results and a  preliminary lemma}
Following the same truncation, centralization and rescaling steps in \cite{BaiS2010}, in this section  we may assume the random variables satisfy the  conditions as follows
  \begin{align*}
    |x_{ij}|\leq n^{1/4},~~\E x_{ij}=0,~~\E|x_{ij}|^2=1 \mbox{ for all  } i,j.
\end{align*}
Bai in \cite{Bai1993} derived
\begin{align}
        s_n(z)&=\frac{1}{n} tr\bfD^{-1}=\frac{1}{n}\sum_{i=1}^n\b_i=-a_n+\frac{a_n}{n}\sum_{i=1}^n\b_i\ep_i\lb{sn}.
\end{align}
For each $i$ we have
\begin{align*}
    |\Im \b_i^{-1}|=|\Im\(z+n^{-1}\bfa_i^*\bfD^{-1}_i\bfa_i\)|\geq v,
\end{align*}
which implies
\begin{align}
|\b_i|\leq v^{-1}.\lb{beta}
\end{align}
From the definition of $\ep_i$ it follows that
\begin{align}
\ep_i=n^{-1/2}x_{ii}-n^{-1}\ga_i+n^{-1}\xi_i-(s_n-\E s_n)\lb{ep},
\end{align}
and
\begin{align}
    s_n=-a_n+\frac{a_n}{n^{3/2}}\sum_{i=1}^n\b_ix_{ii}+\frac{a_n}{n^{2}}\sum_{i=1}^n\b_i\ga_i+\frac{a_n}{n^{2}}\sum_{i=1}^n\b_i\xi_{i}-a_n(s_n-\E s_n)s_n\lb{sn1}.
\end{align}
Then, we  have  the  the following lemma.
\begin{lemma}\lb{3.2}
   Under the assumption in Theorem \ref{th1}, we have
 \begin{align}
    \P\(|\b_i|>2\)\leq\frac{C}{n^{2}v^{2}}\lb{pbeta}.
\end{align}
\end{lemma}
\begin{proof}
   From integration by parts and  Theorem 1.1 in \cite{GotzeT2003}, we have for $1>v>v_0$,
\begin{align*}
    &|\E s_n(z)-s(z)| =\left|\int_{-\infty}^{\infty}\frac{d(\E F^{\mathbf{W}_n}(x)-F^{}(x))}{x-z}\right|\\
    &= \left|\int_{-\infty}^{\infty}\frac{\E F^{\mathbf{W}_n}(x)-F^{}(x)}{(x-z)^2}dx\right|\leq C,
\end{align*}
which together with the fact that $ |s(z)|\leq 1$ ( see (3.3) in \cite{Bai1993}) implies
\begin{align*}
    \E|s_n(z)|\leq C.
\end{align*}
Then, from Lemma \ref{4.1}, Lemma \ref{4.2} and Lemma \ref{4.3}, we can check that
\begin{align}
    \E|\ga_i|^4&\leq C\E\(\(tr\bfD_i^{-1}(\bfD_i^{-1})^*\)^2+n^{1/2}tr\(\bfD_i^{-1}(\bfD_i^{-1})^*\)^2\)\no\\
    &\leq C\(v^{-2}\E|tr\bfD_i^{-1}|^2+n^{1/2}v^{-3}\E|tr\bfD_i^{-1}|\) \no\\
    &\leq \frac{C n^2}{v^2} \lb{ga4}.
\end{align}
Thus, from (\ref{ep}), Lemma \ref{4.2} and Lemma \ref{4.3} we have  for $v>v_0$,
\begin{align}
    \E|\ep_i|^4\leq \frac{C}{n^2v^2}\lb{epi}.
\end{align}

In addition, from (8.1.19) in \cite{BaiS2010}, we know that
 \begin{align}
 |a_n|<1 \mbox{ for all }z\in\mathbb{C}^+ \lb{an}.
\end{align}
Therefore we obtain
\begin{align*}
    \P\(|\b_i|>2\)\leq \P\(|a_n\ep_i|>\frac{1}{2}\)\leq 2^{4}\E|\ep_i|^{4}\leq \frac{C}{n^2v^2}.
\end{align*}
\end{proof}
\subsection{The proof of Lemma \ref{l1}}
Notice that in this subsection, we will use the equality $\b_i=-a_n+a_n\b_i\ep_i$ frequently.
From \eqref{sn1}, we have
\begin{align*}
&\E\lj s_n-\E s_n \rj^2=\E(\ol {s_n-\E s_n})( s_n-\E s_n )\\
=&\E(\ol {s_n-s_n}) s_n=a_n(S_1+S_2+S_3+S_4)
\end{align*}
where
\begin{align*}
    S_1&=\frac{1}{n^{3/2}}\sum_{i=1}^n\E(\ol {s_n-\E s_n})x_{ii}\b_i\\
    S_2&=-\frac{1}{n^{2}}\sum_{i=1}^n\E(\ol {s_n-\E s_n})\ga_i\b_i\\    S_3&=\frac{1}{n^{2}}\sum_{i=1}^n\E(\ol {s_n-\E s_n})\xi_{i}\b_i\\
     S_4&=-\E|s_n-\E s_n|^2s_n.
\end{align*}

We  first consider $S_1$. From \eqref{sn},  we have
\begin{align*}
    S_1&=\frac{1}{n^{3/2}}\sum_{i=1}^n\E(\ol {s_n-\E s_n})x_{ii}\b_i\\
    &=-\frac{a_n}{n^{3/2}}\sum_{i=1}^n\E(\ol {s_n-\E s_n})x_{ii}+\frac{a_n}{n^{3/2}}\sum_{i=1}^n\E(\ol {s_n-\E s_n})x_{ii}\b_i\ep_i\\
    &=S_{11}+S_{12}.
\end{align*}
By \eqref{an} and Lemma \ref{4.2} we  have
\begin{align}
    |S_{11}|=    \lj\frac{a_n}{n^{5/2}}\sum_{i=1}^n\E\xi_ix_{ii}\rj\leq\lj\frac{a_n}{n^{5/2}v}\sum_{i=1}^n\E|x_{ii}|\rj\leq \frac{1}{n^{3/2}v}\lb{3.10}.
\end{align}
Applying \eqref{ep}, we obtain
\begin{align*}
    S_{12}&=\frac{a_n}{n^{3/2}}\sum_{i=1}^n\E(\ol {s_n-\E s_n})x_{ii}\b_i\ep_i\\
    &=S_{121}+S_{122}+S_{123}+S_{124},
\end{align*}
where
\begin{align*}
    S_{121}&=\frac{a_n}{n^{2}}\sum_{i=1}^n\E(\ol {s_n-\E s_n})x_{ii}^2\b_i\\
    S_{122}&=-\frac{a_n}{n^{5/2}}\sum_{i=1}^n\E(\ol {s_n-\E s_n})x_{ii}\b_i\ga_i\\
    S_{123}&=\frac{a_n}{n^{5/2}}\sum_{i=1}^n\E(\ol {s_n-\E s_n})x_{ii}\b_i\xi_i\\
    S_{124}&=-\frac{a_n}{n^{3/2}}\sum_{i=1}^n\E| s_n-\E s_n|^2x_{ii}\b_i.
\end{align*}
Using Lemma \ref{3.2},  Lemma \ref{4.3},  \eqref{beta} and H\"{o}lder's inequality, we  get
\begin{align}
    \lj S_{121}\rj=&\lj\frac{a_n}{n^{2}}\sum_{i=1}^n\E(\ol {s_n-\E s_n})x_{ii}^2\b_i\rj\no\\
    \leq& \frac{C}{n^{2}}\sum_{i=1}^n\(\E\lj(\ol {s_n-\E s_n})x_{ii}^2\rj+v^{-1}\E\lj(\ol {s_n-\E s_n})x_{ii}^2I(|\beta_i|>2)\rj\)\no\\
    \leq&\frac{C}{n^{2}}\sum_{i=1}^n\(\E\lj(\ol {s_n-\E s_n})x_{ii}^2\rj\)\no\\
   \le&\frac{C}{n^{2}}\sum_{i=1}^n\(\E| s_n-\E s_n|^3\)^{1/3}\(\E|x_{ii}^3|\)^{2/3}= O(\frac{1}{n^2v^{3/2}}).\lb{3.11}
\end{align}
Similarly, since $x_{ii}$ and  $\ga_i$ are independent, then by  Lemma \ref{4.1} and  H\"{o}lder's inequality,  we  have
\begin{align}
    |S_{122}|=&\lj\frac{a_n}{n^{5/2}}\sum_{i=1}^n\E(\ol {s_n-\E s_n})x_{ii}\b_i\ga_i\rj\no\\
    \leq&\frac{C}{n^{5/2}}\sum_{i=1}^n\(\E\lj (s_n-\E s_n)x_{ii}\ga_i\rj\)\no\\
    \leq&\frac{C}{n^{5/2}}\sum_{i=1}^n\( \E|s_n-\E s_n|^2\E|\ga_i|^2\)^{1/2}\no\\
    =&O(\frac{1}{n^2v^2}).\lb{3.12}
\end{align}
Using Lemma \ref{3.2} and Lemma \ref{4.2} again,
\begin{align}
    |S_{123}|=&\lj\frac{a_n}{n^{5/2}}\sum_{i=1}^n\E(\ol {s_n-\E s_n})x_{ii}\b_i\xi_i\rj\no\\
    \leq&\frac{C}{n^{5/2}v}\sum_{i=1}^n\(\E| (s_n-\E s_n)x_{ii}|\)\leq\frac{C}{n^{5/2}v^{5/2}},\lb{3.13}
\end{align}
and
\begin{align}
    |S_{124}|=&\lj\frac{a_n}{n^{3/2}}\sum_{i=1}^n\E| s_n-\E s_n|^2x_{ii}\b_i\rj\no\\
\leq&\frac{C}{n^{3/2}}\sum_{i=1}^n\E |s_n-\E s_n|^2|x_{ii}|\leq\frac{C}{n^{5/2}v^{3}}.\lb{3.14}
\end{align}
Therefore combining inequalities \eqref{3.10}-\eqref{3.14}  we obtain
\begin{align}
    |S_{1}|=O\( \frac{1}{n}\)\lb{s1}.
\end{align}

Furthermore, we have the following expression for $S_2$,
\begin{align*}
    S_2&=-\frac{1}{n^{2}}\sum_{i=1}^n\E(\ol {s_n-\E s_n})\ga_i\b_i\\
    &=\frac{a_n}{n^{2}}\sum_{i=1}^n\E(\ol {s_n-\E s_n})\ga_i-\frac{a_n}{n^{2}}\sum_{i=1}^n\E(\ol {s_n-\E s_n})\ga_i\ep_i\b_i\\
    &=S_{21}+S_{22}+S_{23}+S_{24}+S_{25},
\end{align*}
where
\begin{align*}
    S_{21}&=\frac{a_n}{n^{2}}\sum_{i=1}^n\E(\ol {s_n-n^{-1}tr\bfD_i})\ga_i\\
    S_{22}&=-\frac{a_n}{n^{5/2}}\sum_{i=1}^n\E(\ol {s_n-\E s_n})x_{ii}\ga_i\b_i\\
    S_{23}&=\frac{a_n}{n^{3}}\sum_{i=1}^n\E(\ol {s_n-\E s_n})\b_i\ga_i^2\\
    S_{24}&=-\frac{a_n}{n^{3}}\sum_{i=1}^n\E(\ol {s_n-\E s_n})\ga_i\b_i\xi_i\\
    S_{25}&=\frac{a_n}{n^{2}}\sum_{i=1}^n\E| s_n-\E s_n|^2\ga_i\b_i.
\end{align*}
Here we use the method which we used to handle the bound of $S_1$. Firstly, we express $S_{21}$ as follows
  \begin{align*}
    S_{21}&=\frac{a_n}{n^{3}}\sum_{i=1}^n\E(\ol {(1+n^{-1}\bfa_i^*\bfD_i^{-2}\bfa_i)\b_i})\ga_i\\
    &=S_{211}+S_{212},
\end{align*}
where
\begin{align*}
    S_{211}&=-\frac{|a_n|^2}{n^{4}}\sum_{i=1}^n\E(\ol {\hat{\ga}_i})\ga_i\\
    S_{212}&=\frac{|a_n|^2}{n^{3}}\sum_{i=1}^n\E(\ol {(1+n^{-1}\bfa_i^*\bfD_i^{-2}\bfa_i)\b_i\ep_i})\ga_i.
\end{align*}
From Lemma \ref{4.1} and\hi   we get
\begin{align*}
    |S_{211}|\leq \frac{C}{n^{4}}\sum_{i=1}^n\(\E|\hat{\ga}_i|^2\E|\ga_i|^2\)^{1/2}\leq\frac{C}{n^2v^2}.
\end{align*}
Applying Lemma \ref{4.2},\hi and \eqref{epi}, we  obtain \begin{align*}
    S_{212}&=\frac{|a_n|^2}{n^{2}}\lj\sum_{i=1}^n\E(\ol {s_n-n^{-1}tr\bfD_i^{-1}\ep_i})\ga_i\rj\leq \frac{C}{n^{2}v}\sum_{i=1}^n\(\E|\ep_i|^2)\E|\ga_i|^2\)^{1/2}\leq\frac{C}{n^2v^2}.
\end{align*}
Note that $|S_{22}|=|S_{122}|=O(n^{-2}v^{-2}) $. And using Lemma \ref{3.2}, \eqref{ga4} and H\"{o}lder's inequality,  we have
\begin{align*}
    |S_{23}|\leq& \frac{C}{n^{3}}\sum_{i=1}^n\(\E\lj(\ol {s_n-\E s_n})\ga_i^2\rj\)   \leq\frac{C}{nv}\(\E\lj {s_n-\E s_n}\rj^{2}\)^{1/2},
\end{align*}
and
\begin{align*}
       |S_{24}|\leq&\frac{C }{n^{3}v}\sum_{i=1}^n\(\E|s_n-\E s_n|^2)\E|\ga_i|^2\)^{1/2}\leq \frac{C}{n^{5/2}v^{5/2}}.\\
\end{align*}
Now consider  $S_{25}$, using Lemma \ref{3.2}, Lemma \ref{4.3},\hi and \eqref{epi}, we write \begin{align*}
    |S_{25}|=&\frac{|a_n|}{n^{4}}\lj\sum_{i=1}^n\E| tr\bfD_i^{-1}-\E tr\bfD_i^{-1}|^2\ga_i\b_i\rj+O(\frac{1}{n^{5/2}v^{5/2}})\\
    =&\frac{|a_n|^2}{n^{4}}\lj\sum_{i=1}^n\E| tr\bfD_i^{-1}-\E tr\bfD_i^{-1}|^2\ga_i\ep_i\b_i\rj+O(\frac{1}{n^{5/2}v^{5/2}})\\
     \leq&\frac{C}{n^{4}}\sum_{i=1}^n\(\E| tr\bfD_i^{-1}-\E tr\bfD_i^{-1}|^4\E|\ep_i|^4\(\E|\ga_i|^2\)^2\)^{1/4}+O(\frac{1}{n^{5/2}v^{5/2}})\\
     =&O\(\frac{1}{n^2v^2}\).
\end{align*}
Then, we conclude that
\begin{align}
    |S_2|\leq\frac{C}{nv}\(\E\lj {s_n-\E s_n}\rj^{2}\)^{1/2} +\frac{C}{n^2v^2}\lb{s2}.
\end{align}

From Lemma \ref{3.2}, Lemma \ref{4.2} and \hi, it is easy to check that
\begin{align}
    |S_3|\leq \frac{C}{nv}\(\E|s_{n}-\E s_{n}|^2\)^{1/2}\lb{s3}.
\end{align}
Therefore, it remians to get the bound of $S_4$. Now we recall the equality \eqref{sn1},
\begin{align*}
 S_4&=-\E|s_n-\E s_n|^2s_n\no\\
     &=-\E s_n\E|s_n-\E s_n|^2-\E|s_n-\E s_n|^2(s_n-\E s_n)\\
     &=-\E s_n\E|s_n-\E s_n|^2+(a_n+\E s_n)\E|s_n-\E s_n|^2-\frac{a_n}{n}\sum_{i=1}^n\E|s_n-\E s_n|^2\b_i\ep_i\no\\
     &=-\E s_n\E|s_n-\E s_n|^2+(a_n+\E s_n)\E|s_n-\E s_n|^2-a_n(S_{41}+S_{42}+S_{43}+S_{44}),\no
\end{align*}
where
\begin{align*}
    S_{41}=&\frac{1}{n^{3/2}}\sum_{i=1}^n\E|s_n-\E s_n|^2x_{ii}\b_i\no\\
    S_{42}=&-\frac{1}{n^{2}}\sum_{i=1}^n\E|s_n-\E s_n|^2\ga_i\b_i\no\\    S_{43}=&\frac{1}{n^{2}}\sum_{i=1}^n\E|s_n-\E s_n|^2\xi_{i}\b_i\no\\
     S_{44}=&-\E|s_n-\E s_n|^2(s_n-\E s_n)s_n\no\\
     =&-\E s_n\E|s_n-\E s_n|^2(s_n-\E s_n)\\
     &-\E|s_n-\E s_n|^2(s_n-\E s_n)^2.
\end{align*}
Comparing $S_4$ with $S_{44}$,  we obtain that
\begin{align*}
   &(1+a_n\E s_n)\E|s_n-\E s_n|^2(s_n-\E s_n)\\
   =&-(a_n+\E s_n)\E|s_n-\E s_n|^2\\
   &+a_n(S_{41}+S_{42}+S_{43}-\E|s_n-\E s_n|^2(s_n-\E s_n)^2),
\end{align*}
which implies that
\begin{align*}
    &-\E|s_n-\E s_n|^2(s_n-\E s_n)\\
     =&b_na^{-1}_n(a_n+\E s_n)\E|s_n-\E s_n|^2\\
     &-b_n(S_{41}+S_{42}+S_{43}-\E|s_n-\E s_n|^2(s_n-\E s_n)^2)
\end{align*}
Thus denote $\de_n=n^{-1}\sum_{i=1}^n\E\b_i\ep_i$, we conclude that
\begin{align*}
    S_4=&(-\E s_n+b_na^{-1}_n(a_n+\E s_n))\E|s_n-\E s_n|^2\\
    &-b_n(S_{41}+S_{42}+S_{43}-\E|s_n-\E s_n|^2(s_n-\E s_n)^2)\\
    =&(a_n-\de_nb_n\E s_n)\E|s_n-\E s_n|^2\\
    &-b_n(S_{41}+S_{42}+S_{43}-\E|s_n-\E s_n|^2(s_n-\E s_n)^2)\\
        =&(a_n+a_n\de_nb_n)\E|s_n-\E s_n|^2\\
        &-b_n(\de_n^2\E|s_n-\E s_n|^2+S_{41}+S_{42}+S_{43}-\E|s_n-\E s_n|^2(s_n-\E s_n)^2).
\end{align*}
It is obvious that $S_{41}$ and $S_{124}$ have the  same  bound, $S_{42}$ and  $S_{25}$  have the  same  bound. Using Lemma \ref{4.2} and Lemma \ref{4.3} we get
 \begin{align*}
    |\E|s_n-\E s_n|^2(s_n-\E s_n)^2|\leq \E|s_n-\E s_n|^4\leq \frac{C}{n^4v^6},
\end{align*}
and
\begin{align*}
    |S_{43}|\leq \frac{1}{nv}\(\E|s_n - \E s_n|^4\)^{1/2}\leq \frac{C}{n^3v^4}.
\end{align*}
Furthermore, from the definition of $\de_i$ and \eqref{epi}, we have
\begin{align*}
    |\de_n|=\lj n^{-1}\sum_{i=1}^n\(\E n^{-1}\bfD_i^{-1}-\E s_n+\E\b_i\ep_i^2\)\rj\leq\frac{C}{nv}.
\end{align*}
Therefore, we obtain
\begin{align*}
    S_4=a_n\E|s_n-\E s_n|^2+O\(\frac{|b_n|}{n^2v^2}\),
\end{align*}
which  combined with \eqref{s1},\eqref{s2} and \eqref{s3} implies
\begin{align*}
    |1-a_n^2|\E|s_n-\E s_n|^2\leq \frac{C_1|a_nb_n|}{n}+\frac{C_2|a_n|}{\sqrt{n}}\(\E|s_n-\E s_n|^2\)^{1/2}.
\end{align*}
Then, from (6.91) and (6.95) in \cite{GotzeT2003} which are under existing fourth moment assumption,  for $1>v>v_0$,
\begin{align*}
   |1-a_n^2|\geq|a_n(z+2s(z))|\mbox{ and } |b_n|\leq2|z+2s(z)|^{-1},
\end{align*}
we obtain the following inequality
\begin{align*}
    \E|s_n-\E s_n|^2\leq \frac{C_1}{n|z+2s(z)|^2}+\frac{C_2}{\sqrt{n}|z+2s(z)|}\(\E|s_n-\E s_n|^2\)^{1/2}.
\end{align*}
Solving this inequality, we obtain
\begin{align*}
    \E|s_n-\E s_n|^2\leq \frac{C}{n|z+2s(z)|^2},
\end{align*}
which complete the proof of the Lemma.

\section{Basic lemmas}
In this section we list some results which are needed in the proof.
\begin{lemma}\label{4.1}(Lemma B.26 of \cite{BaiS2010}) Let $\bfA$ be an $n \times n$ nonrandom matrix and $\bfX = (x_1,\dots, x_n)^*$ be a random vector of independent entries. Assume that $\E x_i = 0$, $\E |x_i|^2 = 1$, and $E|x_j |^l\leq\nu_l$. Then, for any $p \geq 1$,
  \begin{align*}
    \E|\bfX^*\bfA \bfX-tr\bfA|^p\leq C_p\(\(\nu_4tr(\bfA\bfA^*)\)^{p/2}+\nu_{2p}tr(\bfA\bfA^*)^{p/2}\),
  \end{align*}
  where $C_p$ is a constant depending on $p$ only.
\end{lemma}
\begin{lemma}\label{4.2}(Lemma 2.6 of \cite{SilversteinB1995}). Let  $z\in\mathbb{C}^+$ with $v=\Im z$, $\bfA$ and $\mathbf{B}$ $n\times  n$ with $\mathbf{B}$ Hermitian, $\tau\in\mathbb{R}$, and $\mathbf{q}\in\mathbb{C}^N$. Then
\begin{align*}
    |tr((\mathbf{B}-z\bfI)^{-1}-(\mathbf{B}+\tau \mathbf{q}\mathbf{q}^*-z\bfI)^{-1})\bfA|\leq\frac{\|\bfA\|}{v}.
\end{align*}
\end{lemma}
\begin{lemma}(Lemma 8.7 of  \cite{BaiS2010})\lb{4.3}
   Under the assumption in Theorem \ref{th1}, we have
   \begin{align}
    \E|s_n(z)-\E s_n(z)|^{2l}\leq \frac{C}{n^{2l}v^{3l}}.
   \end{align}
\end{lemma}

\end{document}